\documentclass[11pt,a4paper]{amsart}
\usepackage{amssymb,amsmath,amsfonts}

\headsep=1cm \numberwithin{equation}{section}
\newtheorem{theorem}{Theorem}[section]
\newtheorem{lemma}[theorem]{Lemma}
\newtheorem{proposition}[theorem]{Proposition}
\newtheorem{corollary}[theorem]{Corollary}

\theoremstyle{definition}

\theoremstyle{notation}
\newtheorem{notation}[theorem]{Notation}
\theoremstyle{remark}

\newcommand{\fm}{\mathfrak{m}}

\newcommand{\fp}{\frak{p}}

\newcommand{\fa}{\frak{a}}
\newcommand{\fb}{\frak{b}}

\title[on vanishing and cofiniteness of generalized local cohomology]
 {on vanishing and cofiniteness of generalized local cohomology modules}

    \author[S. H. Hassanzadeh]{S. H. Hassanzadeh $^{1}$}
\address{$^{1}$Institut de mathematiques, Universite Paris 6,  4,
Place Jussieu, F-75252 Paris, Cedex 05, France. And Faculty of
Mathematical science and computer engineering, Teacher Training
University, 599 Taleghani Avenue, 1561836314 Tehran, Iran; Fax:
+98(021)77602988. } \email{h\_hassanzadeh@tmu.ac.ir}

    \author[A. Vahidi]{A. Vahidi $^{2}$}

 \address{$^{2}$Faculty of Mathematical science and computer engineering,
Teacher Training University, 599 Taleghani Avenue, 1561836314
Tehran, Iran; Fax: +98(021)77602988.} \email{vahidi.ar@gmail.com}

\dedicatory{The second author would like to dedicate this paper\\
to the dear and loving memory of his late mother\\
who battled cancer and lost the fight in October 2007.}

\keywords{Generalized local cohomology, Cofiniteness, Spectral
sequence}

\subjclass[2000]{13D45, 14B15}

\begin{document}
\maketitle

\begin{abstract}

In this paper,  some results on vanishing and non-vanishing of
generalized local cohomology modules are presented and some
relations between those modules and,  $Ext$ and ordinary local
cohomology modules are studied. Also, several cofiniteness
propositions for generalized local cohomology modules are
established which, among other things, provide an alternative answer
to a question in [Y2].
\end{abstract}




\section{Introduction}
The notion of generalized local cohomology was first introduced by
J. Herzog in his habilitationss [H] and then continued by N. Suzuki
[S], S. Yassemi [Y] and some other authors. They studied some basic
duality theorems, vanishing and other properties of generalized
local cohomology modules which also generalize several known facts
about $Ext$ and ordinary local cohomolgy modules.

Throughout the paper, $R$ is a commutative Noetherian ring  with
identity, $\frak{a}$ is an ideal of $R$. 
 For each $i\geq 0,$ the generalized local cohomology functor $H^{i}_{\frak{a}}(-,-)$ is
defined, for all $R$-modules $M$ and $N$, by
$$H^{i}_{\frak{a}}(M,N)=\displaystyle\lim_{\stackrel{\longrightarrow}{n}}Ext^{i}_{R}(M/\frak{a}^{n}M,N).$$

In section 2, we study the vanishing and non-vanishing of
generalized local cohomology modules. The crucial point is Theorem
2.2 which is interesting in itself as it provides a vanishing
proposition for $Ext$ and ordinary local cohomology modules. Then
Proposition 2.3 and Corollary 2.4 show what implied from this
theorem about vanishing and non-vanishing  of generalized local
cohomology modules. In Proposition 2.6, we extend Theorem 2.2 for
Artinian property of local cohomology modules which gives a new
proof for [Mel; Theorem 5.5]. Finally, in this section, Proposition
2.8 shows how the generalized local cohomology modules behave
sharply at the non-vanishing bounds for $Ext$ and local cohomology
modules. We also apply this proposition in section 3.

In section 3, we study the cofiniteness property of generalized
local cohomology modules. In the first proposition we show that for
finite $R$-modules $M$ and $N$ with $pd(M)<\infty$, the $R$-module
$H^{pd(M)+dim(N)}_{\frak{a}}(M,N)$ is not only Artinian, as shown by
A. Mafi [M; Theorem 2.9], but also $\frak{a}$-cofinite. Then in
Proposition 3.4 we generalize [Mel; Corollary 2.5] for a Serre
subcategory of the category of $R$-modules. Furthermore we
generalize and give a new proof to [Mel; Corollary 3.14] for
arbitrary $R$-modules in Proposition 3.5. These two propositions can
also be employed to prove Proposition 3.6 which presents some finite
generalized local cohomology modules. The main theorem of section 3
is Theorem 3.7 which provides some sufficient conditions for
generalized local cohomology modules to be $\frak{a}$-cofinite. Even
though the assumptions in this theorem are apparently strong, they
are practically useful as we can see in the next corollaries which
demonstrate some new facts and represent some older facts about
$\frak{a}$-cofiniteness of generalized local cohomology modules.



\section{Vanishing and non-vanishing results}

In this section, we investigate some basic properties  of
generalized local cohomology modules. Most properties of generalized
local cohomology modules inherit from the spectral sequence
$$E^{p,q}_{2}:=Ext^{p}_{R}(M,H^{q}_{\frak{a}}(X))
_{\stackrel{\Longrightarrow}{p}} H^{p
+q}_{\frak{a}}(M,X)\hspace{4cm}(*)$$ in particular, when we put this
spectral sequence in a double complex. All of the facts in this
section can be proved by means of the method that we apply in
Theorem 2.2 but to shorten the other proofs, we just mention the
corresponding spectral sequence's property.

 The following lemma is an immediate
consequence of $(*)$ or can be drawn directly, applying the basic
properties of generalized local cohomology such as [DST; Lemma
2.1(i)].


\begin{lemma} Suppose that $M$ is a finite $R$-module and that $X$ is an arbitrary
$R$-module. Then
\begin{enumerate}
  \item[(i)]{$\Gamma_{\frak{a}}(M,X)=0$ whenever $\Gamma_{\frak{a}}(X)=0$}.
  \item[(ii)]{$\Gamma_{\frak{a}}(M,X)=H^{1}_{\frak{a}}(M,X)=0$  whenever
  $\Gamma_{\frak{a}}(X)=H^{1}_{\frak{a}}(X)=0$}.
\end{enumerate}
\end{lemma}


\begin{theorem}
Let $X$ be an arbitrary $R$-module (not necessarily finite) and $t$
be a positive integer. Then the following statements are equivalent:
       \begin{enumerate}
          \item[(i)]{$H^{i}_{\frak{a}}(X)=0$ for all $0 \leq i <t.$}
          \item[(ii)]{$Ext^{i}_{R}(R/\frak{a},X)=0$ for all $0 \leq i < t.$}
       \end{enumerate}
\end{theorem}

\begin{proof}

 We prove by using induction on t. Let $t=1$. Consider the
isomorphism $Hom_{R}(R/\frak{a},X)\cong
Hom_{R}(R/\frak{a},\Gamma_{\frak{a}}(X))$; so that we have an
equivalent condition that $Hom_{R}(R/\frak{a},X)=0$ if and only if
$\Gamma_{\frak{a}}(X)=0,$ since $Supp(\Gamma_{\frak{a}}(X))\subseteq
V(\fa).$ Hence the assertion holds in this case.

Suppose that $t>1$ and that we have proved the theorem for $t-1$ and
all $R$-modules $X$. Let $Y$ be an $R$-module and consider the exact
sequence
$$0 \rightarrow Y \rightarrow E(Y) \rightarrow L \rightarrow 0,$$
 where $E(Y)$ is an injective hull of $Y$ and $L$ is an $R$-module.
 Applying the derived functors of $Hom_{R}(R/\frak{a},-)$ and $\Gamma_{\frak{a}}(-)$ to the above short exact
 sequence, we have exact sequences and isomorphisms:
$$0 \rightarrow Hom_{R}(R/\frak{a},Y) \rightarrow Hom_{R}(R/\frak{a},E(Y))\rightarrow
Hom_{R}(R/\frak{a},L)\rightarrow
Ext^{1}_{R}(R/\frak{a},Y)\rightarrow 0,$$
$$Ext^{i-1}_{R}(R/\fa,L)\cong Ext^{i}_{R}(R/\fa,Y)$$
for all $i>1,$ and
$$0\rightarrow \Gamma_{\frak{a}}(Y)\rightarrow
\Gamma_{\frak{a}}(E(Y))\rightarrow \Gamma_{\frak{a}}(L)\rightarrow
H^{1}_{\frak{a}}(Y)\rightarrow 0,$$
$$H^{i-1}_{\fa}(L)\cong H^{i}_{\fa}(Y)$$
for all $i>1.$

Now, assume that $H^{i}_{\fa}(Y)=0$ for all $0 \leq i < t$. Since
$\Gamma_{\frak{a}}(Y)=0$, we have $\Gamma_{\frak{a}}(E(Y))=0$ from
the fact that $E(Y)$ is an essential extension of $Y$; so that
$Hom_{R}(R/\frak{a},Y) \subseteq Hom_{R}(R/\frak{a},E(Y))=0$. Hence
we have, by the above isomorphisms and exact sequences,
$R$-isomorphisms:
$$Ext^{i-1}_{R}(R/\fa,L)\cong Ext^{i}_{R}(R/\fa,Y)\quad and \quad
H^{i-1}_{\fa}(L)\cong H^{i}_{\fa}(Y)$$ for all $i>0.$

Now, $H^{i}_{\fa}(L)=0$ for all $0\leq i <t-1$ form the latter
isomorphisms. Hence $Ext^{i}_{R}(R/\fa,L)=0$ for all $0\leq i <t-1$
by the induction hypothesis on $L$. Again from the above
isomorphisms, $Ext^{i}_{R}(R/\fa,Y)=0$ for all $0\leq i <t$.

Conversely, as the above proof, we see that if
$Ext^{i}_{R}(R/\fa,Y)=0$ for all $0\leq i <t$, then
$H^{i}_{\fa}(Y)=0$ for all $0 \leq i <t.$ We leave the proof to the
reader. The proof is completed.
\end{proof}

\begin{proposition}
 Let M be a finitely generated R-module, X be an R-module, and t be a positive integer.
  If $R$-module $X$ satisfies the equivalent conditions in Theorem 2.2 for the integer $t$,
then   we have
       \begin{enumerate}
          \item[(i)]{$H^{i}_{\frak{a}}(M,X)=0$ for all $0 \leq i < t.$}
          \item[(ii)]{$H^{t}_{\frak{a}}(M,X)\cong Hom_{R}(M,H^{t}_{\frak{a}}(X))$.}
          \item[(iii)]{$Ext^{t}_{R}(R/\frak{a},X)\cong Hom_{R}(R/\frak{a},H^{t}_{\frak{a}}(X)).$}
       \end{enumerate}
\end{proposition}

\begin{proof}
Before beginning the proof, recall that by $D_{\frak{a}}(X)$ we mean
the ideal transform of $X$ with respect to $\fa$ which is defined by
$D_{\frak{a}}(X):=\displaystyle\lim_{\stackrel{\longrightarrow}{n}}Hom_{R}(\frak{a}^{n},X)$.
For the basic definitions and theorems about local cohomology and
ideal transform we refer the reader to [BS].

We use induction on $t$. Let $t=1$. Since $\Gamma_{\fa}(X)=0$,
$\Gamma_{\fa}(M,X)=0$ by Lemma 2.1(i). We have
$\Gamma_{\fa}(D_{\fa}(X))=H^{1}_{\fa}(D_{\fa}(X))=0$ by [BS;
Corollary 2.2.8(iv)]. Therefore
$\Gamma_{\fa}(M,D_{\fa}(X))=H^{1}_{\fa}(M,D_{\fa}(X))=0$ by Lemma
2.1(ii) and,
$Hom_{R}(R/\fa,D_{\fa}(X))=Ext^{1}_{R}(R/\fa,D_{\fa}(X))=0$ by
Theorem 2.2; so that the ideal transform sequence $0 \rightarrow X
\rightarrow D_{\fa}(X) \rightarrow H^{1}_{\fa}(X) \rightarrow 0,$
yields $Ext^{1}_{R}(R/\fa,X)\cong Hom_{R}(R/\fa,H^{1}_{\fa}(X))$ and
$H^{1}_{\fa}(M,X)\cong \Gamma_{\fa}(M,H^{1}_{\fa}(X))$. Now, the
latter isomorphisms in conjunction with [YKS; Lemma 1.1] imply the
assertion in the case where $t=1$.

Let $Y$ be an $R$-module and $t>1$. Also suppose that our claims are
satisfied  for $t-1$ and all $R$-modules $X$.

Now, assume that $H^{i}_{\fa}(Y)=0$ for all $0 \leq i < t$. Since
$\Gamma_{\frak{a}}(Y)=0$, we have $\Gamma_{\frak{a}}(E(Y))=0$ from
the fact that $E(Y)$ is an essential extension of $Y$; so that
$\Gamma_{\frak{a}}(M,E(Y))=0$ by Lemma 2.1(i). Using the exact
sequence $0 \rightarrow Y \rightarrow E(Y) \rightarrow L \rightarrow
0$ as we used in the above theorem and applying the derived functors
of $Hom_{R}(R/\frak{a},-), \ \Gamma_{\frak{a}}(M,-)$ and
$\Gamma_{\frak{a}}(-)$ to this short exact
 sequence, we obtain the isomorphisms:
$$Ext^{i-1}_{R}(R/\fa,L)\cong Ext^{i}_{R}(R/\fa,Y)\ ,\
H^{i-1}_{\fa}(M,L)\cong H^{i}_{\fa}(M,Y)\  and \
H^{i-1}_{\fa}(L)\cong H^{i}_{\fa}(Y)$$ for all $i > 0$.

Now, our assumption on $Y$ and  the latter isomorphisms show that,
$H^{i}_{\fa}(L)=0$ for all $0\leq i
 <t-1$. Hence $H^{t-1}_{\frak{a}}(M,L)\cong
 Hom_{R}(M,H^{t-1}_{\frak{a}}(L)),Ext^{t-1}_{R}(R/\frak{a},L)\cong
 Hom_{R}(R/\frak{a},H^{t-1}_{\frak{a}}(L))$ and
 $H^{i}_{\frak{a}}(M,L)=0$ for all $0 \leq i < t-1$ by the
induction hypothesis on $L$. Again from the above isomorphisms, we
have $H^{t}_{\frak{a}}(M,Y)\cong Hom_{R}(M,H^{t}_{\frak{a}}(Y)),
Ext^{t}_{R}(R/\frak{a},Y)\cong
Hom_{R}(R/\frak{a},H^{t}_{\frak{a}}(Y))$ and
$H^{i}_{\frak{a}}(M,Y)=0$ for all $0 \leq i < t,$ as we desired.
\end{proof}

The following corollary can be obtained, by straightforward
applications, from Theorem 2.2 and Proposition 2.3.

\begin{corollary} Let $M,N$ be non-zero finite
$R$-modules, $X$ be an arbitrary $R$-module, and $t$ be  a positive
integer. Then
\begin{enumerate}
  \item[(i)]{
               \begin{enumerate}
                   \item{$Ext^{1}_{R}(R/\fa,D_{\fa}(X))=H^{1}_{\fa}(M,D_{\fa}(X))=0$.}
                   \item{$Ext^{2}_{R}(R/\fa,D_{\fa}(X))\cong Hom_{R}(R/\fa,H^{2}_{\fa}(X))$ \\and
                         $H^{2}_{\fa}(M,D_{\fa}(X))\cong Hom_{R}(M,H^{2}_{\fa}(X))$.}
                   \item{$Ext^{1}_{R}(R/\fa,X/\Gamma_{\fa}(X))\cong Hom_{R}(R/\fa,H^{1}_{\fa}(X))$ \\and
                         $H^{1}_{\fa}(M,X/\Gamma_{\fa}(X))\cong Hom_{R}(M,H^{1}_{\fa}(X))$.}
               \end{enumerate} }
  \item[(ii)]{If $H^{i}_{\fa}(X)=0$ for all $0 \leq i < t$, then $H^{i}_{\fa}(M,X)=0$ for all $0 \leq i <
  t$.\\
  In this case $H^{t}_{\fa}(M,X)\neq 0$ if and only if $Supp(M) \cap Ass(H^{t}_{\fa}(X))\neq \O$}.
  \item[(iii)]{If a maximal ideal $\fm$ belongs to $Supp(M)$ and $t=grade(\fm,N)$, then
$H^{t}_{\fm}(M,N)\neq
  0$}.
  \item[(iv)]{(c.f.[S; Theorem 2.3]) If $(R,\fm)$ is a local ring and
  $t = depth(N)$, then $H^{t}_{\fm}(M,N)\neq 0$.}
\end{enumerate}
\end{corollary}


\begin{lemma}
Suppose that  $P$ is a finite projective $R$-module and $X$ is an
arbitrary $R$-module. Then $H^{i}_{\fa}(P,X)\cong
Hom_{R}(P,H^{i}_{\fa}(X))$ for all $i\geq 0$.
\end{lemma}
\begin{proof}
This lemma can be proved  by using the method implied in Proposition
2.3. It is also an immediate consequence of the spectral sequence
$$E^{i,j}_{2}:=Ext^{i}_{R}(P,H^{j}_{\frak{a}}(X))
_{\stackrel{\Longrightarrow}{i}} H^{i+j}_{\frak{a}}(P,X)$$ where in
the corresponding double complex $(E^{i,j}_{2})_{i\geq 0 ,j\geq 0}$
all of the columns except the first one vanish.
\end{proof}

In the course of the remaining parts of the paper for  an arbitrary
$R$-module $X$, by $cd(\fa,X)$ (cohomological dimension of $X$ at
$\fa$) we mean the largest integer $i$ in which $H^{i}_{\fa}(X)$ is
non-zero. Also $pd(X)$ and $dim(X)$ are denoted as the  projective
dimension and the Krull dimension of $X$, respectively. The next
proposition shows how Artinian property behaves similarly at the
initial points of $Ext$ and at those of ordinary local cohomology
modules. This proposition gives a new proof for [Mel; Theorem 5.5].

\begin{proposition}
Let  $X$ be  an arbitrary $R$-module  and  $t$ be a positive
integer. Then the following statements are equivalent:
       \begin{enumerate}
          \item[(i)]{$H^{i}_{\fa}(X)$ is Artinian  for all $0 \leq i <
  t.$}
          \item[(ii)]{$Ext^{i}_{R}(R/\fa,X)$ is Artinian  for all $0 \leq i < t.$ }
       \end{enumerate}
\end{proposition}

\begin{proof}

 We prove by induction on $t$. Considering
$0:_{X}\fa=0:_{\Gamma_{\fa}(X)}\fa$ and applying Melkersson's
Theorem [BS; Theorem 7.1.2], the case where $t=1$ is
straightforward.

Let $Y$ be an $R$-module and $t>1$. Also assume that we have proved
the claim for $t-1$ and all $R$-modules $X$. Let
$N=Y/\Gamma_{\frak{a}}(Y)$ and $L=E(N)/N$ where $E(N)$ is an
injective hull of $N$. Since $\Gamma_{\frak{a}}(N)=0,$ we have
$\Gamma_{\frak{a}}(E(N))=0$ from the fact that $E(N)$ is an
essential extension of $N$. Now consider the exact sequence
$$0 \rightarrow N \rightarrow E(N) \rightarrow L \rightarrow 0.$$
 Applying the derived functors of $Hom_{R}(R/\frak{a},-)$ and $\Gamma_{\frak{a}}(-)$ to the above short exact
 sequence, we have isomorphisms:
$$Ext^{i-1}_{R}(R/\fa,L)\cong Ext^{i}_{R}(R/\fa,N) \quad and \quad H^{i-1}_{\fa}(L)\cong H^{i}_{\fa}(N)$$
for all $i>0.$

Now, assume that $H^{i}_{\fa}(Y)$ is Artinian  for all $0 \leq i <
t.$ From the latter isomorphisms and [BS; Corollary 2.1.7], we have
$H^{i}_{\fa}(L)$ is Artinian  for all $0 \leq i < t-1.$ Hence
$Ext^{i}_{R}(R/\fa,L)$ is Artinian  for all $0 \leq i < t-1$ by the
induction hypothesis on $L$. Again by the above isomorphisms,
$Ext^{i}_{R}(R/\fa,N)$ is Artinian  for all $0 \leq i < t.$ Now, the
exact sequence $0 \rightarrow \Gamma_{\frak{a}}(Y) \rightarrow Y
\rightarrow N \rightarrow 0$ in conjunction with Artinianness of
$\Gamma_{\frak{a}}(Y)$ show that, $Ext^{i}_{R}(R/\fa,Y)$ is Artinian
for all $0 \leq i < t.$

Conversely, as the above proof, we see that if
$Ext^{i}_{R}(R/\fa,Y)$ is Artinian  for all $0 \leq i < t$, then
$H^{i}_{\fa}(Y)$ is Artinian for all $0 \leq i < t.$ We leave the
proof to the reader. The proof is completed.
\end{proof}

\begin{proposition} Let $M$ be a finite $R$-module with $pd(M)<\infty$, $X$ be an arbitrary
$R$-module,  and  $t$ be a positive integer. If $R$-module $X$
satisfies the equivalent conditions in Proposition 2.6 for the
integer $t$, then $H^{i}_{\fa}(M,X)$ is Artinian  for all $\ 0 \leq
i < t.$
\end{proposition}

\begin{proof}
 We use induction  on $pd(M).$ The case where $pd(M)=0$ yields
$M\oplus M'\cong R^{n}$ for some $R-$module $M'$ and some integer
$n$. Then, by Lemma 2.5,
$$H^{i}_{\fa}(M,X)\oplus H^{i}_{\fa}(M',X)\cong Hom_{R}(M,H^{i}_{\fa}(X))\oplus Hom_{R}(M',H^{i}_{\fa}(X))$$
$$ \cong Hom_{R}(R^{n},H^{i}_{\fa}(X))\cong H^{i}_{\fa}(X)^{n}$$
which gives the assertion in this case.

Suppose that $pd(M)>0$ and that the assertion is true for every
finite $R$-module $T$ with $pd(T)<pd(M)$. Taking the exact sequence
$0 \rightarrow T \rightarrow F \rightarrow M \rightarrow 0,$ where
$F$ is  free and $T$ is an $R$-module, we set the long exact
sequence
$$\cdots \rightarrow H^{i-1}_{\fa}(T,X) \rightarrow H^{i}_{\fa}(M,X) \rightarrow H^{i}_{\fa}(F,X) \rightarrow \cdots$$
where by induction hypothesis $H^{i-1}_{\fa}(T,X)$ and
$H^{i}_{\fa}(F,X)$ are Artinian; so that $H^{i}_{\fa}(M,X)$ is
Artinian as desired.
\end{proof}

The below proposition can also be proved either by considering the
spectral sequence $(*)$ or using the method in Proposition 2.3 based
on an induction on $cd(\fa,X).$


\begin{proposition}
(compare [Y; Theorem 2.5]) Let $M$ be a finite $R$-module  with
$pd(M)<\infty$ and $X$ be an arbitrary $R$-module. Then for all
$i>pd(M)+cd(\fa,X)$, $H^{i}_{\fa}(M,X)=0$ and
$$H^{pd(M)+cd(\fa,X)}_{\fa}(M,X)\cong
Ext^{pd(M)}_{R}(M,H^{cd(\fa,X)}_{\fa}(X)).$$
\end{proposition}


\section{Cofiniteness results}
In this section we deal with the cofiniteness property of
generalized local cohomology modules. Recall that an $R$-module $X$
is said to be $\fa$-cofinite if $Supp(X)\subseteq V(\fa)$ and either
$Ext^{i}_{R}(R/\fa,X)$ is finite for all $i$ or equally,
$Tor^{R}_{i}(R/\fa,X)$ is finite for all $i$ [Mel; Theorem 2.1].


\begin{proposition}
(compare [M; Theorem 2.9)]) Assume that $M,N$ are finite $R$-modules
such that $t=pd(M)$ and $d=dim(N)$ are finite. Then
$H^{t+d}_{\fa}(M,N)$ is an $\fa$-cofinite Artinian module.

\begin{proof}
If $cd(\fa,N)<d,$ $H^{t+d}_{\fa}(M,N)=0$ by Proposition 2.8. Hence
suppose that $cd(\fa,N)=d$ then we have $H^{t+d}_{\fa}(M,N)\cong
Ext^{t}_{R}(M,H^{d}_{\fa}(N)),$ again by Proposition 2.8. Now,
according to [Mel; Proposition 5.1] $H^{d}_{\fa}(N)$ is an
$\fa$-cofinite Artinian module and also by [Mel; Corollary 4.4] the
category of $\fa$-cofinite Artinian modules is a Serre subcategory
of the category of $R$-modules, i.e. this subcategory is closed
under taking subobjects, quotients and extensions. Therefore
$Ext^{t}_{R}(M,H^{d}_{\fa}(N))$ and hence $H^{t+d}_{\fa}(M,N)$ is an
$\fa$-cofinite Artinian module.
\end{proof}
\end{proposition}


\begin{notation}
For an ideal $\fb$ of $R$, we denote the category of finite
$R/\fb$-modules by $\mathcal{C}^{f}(R/\fb)$ and the category of
finite $R$-modules whose support contained in $V(\fb)$ by
$\mathcal{C}^{\fb}(R)$. Also we denote the category of $R$-modules
by $\mathcal{C}(R)$.
\end{notation}

Considering the fact that each member of $\mathcal{C}^{\fa}(R)$
belongs to $\mathcal{C}^{f}(R/\fa^{n})$ for some positive integer
$n$ and appling a straightforward induction on $n$, we deduce the
following useful lemma.


\begin{lemma}Suppose that $T$ is a middle exact covariant
(contravariant) functor from $\mathcal{C}(R)$ to $\mathcal{C}(R)$
and that $\delta$ is a Serre subcategory of $\mathcal{C}(R)$. Then
 $T(L)\in \delta$ for each $L \in \mathcal{C}^{\fa}(R)$ whenever
 $T(L)\in \delta$ for each $L \in \mathcal{C}^{f}(R/\fa)$.
\end{lemma}


\begin{proposition}(compare [Mel; Corollary 2.5]) Let $X$ be an arbitrary $R$-module, $t$ be a
non-negative integer, and $\delta$ be a Serre subcategory of
$\mathcal{C}(R)$. Suppose $\fb$ be a second ideal of $R$ contains
$\fa$. Then
\begin{enumerate}
  \item[(i)]{If $Ext^{i}_{R}(R/\fa,X)\in \delta$ for all $0\leq i\leq t$, then
   $Ext^{i}_{R}(N,X)\in \delta$ for all $0\leq i\leq t$ and all $N\in \mathcal{C}^{\fa}(R).$
   In particular,  $Ext^{i}_{R}(R/\fb,X)\in \delta$ for all $0\leq i\leq t$}.
  \item[(ii)]{If $Tor^{R}_{i}(R/\fa,X)\in \delta$ for all $0\leq i\leq t$, then
   $Tor^{R}_{i}(N,X)\in \delta$ for all $0\leq i\leq t$ and all $N\in \mathcal{C}^{\fa}(R).$
   In particular,  $Tor^{R}_{i}(R/\fb,X)\in \delta$ for all $0\leq i\leq t.$
   }
\end{enumerate}

\begin{proof}
We just prove (i), the proof for (ii) is similar. First, we use
induction on $i$ to prove that $Ext^{i}_{R}(N,X)\in \delta$ whenever
$N \in \mathcal{C}^{f}(R/\fa)$. Let $i=0$ and $N \in
\mathcal{C}^{f}(R/\fa)$. Considering a suitable  integer $s$ and an
exact sequence $(R/\fa)^{s} \rightarrow N \rightarrow 0$, we obtain
the exact sequence $0 \rightarrow Hom_{R}(N,X) \rightarrow
Hom_{R}((R/\fa)^{s},X)$ where $Hom_{R}((R/\fa)^{s},X) \in \delta $
by  our assumption; so that $Hom_{R}(N,X) \in \delta$ which
completes the proof for the case where $i=0.$

Suppose that $1\leq i\leq t$ and that our assertion holds for $i-1$.
Let $N \in \mathcal{C}^{f}(R/\fa)$. By an exact sequence $0
\rightarrow K \rightarrow (R/\fa)^{s} \rightarrow N \rightarrow 0,$
where $K\in \mathcal{C}^{f}(R/\fa)$ and $s$ is an integer, we obtain
a long exact sequence
$$\cdots \rightarrow Ext^{i-1}_{R}(K,X) \rightarrow
Ext^{i}_{R}(N,X) \rightarrow Ext^{i}_{R}((R/\fa)^{s},X) \rightarrow
\cdots$$ where by the assumption in (i), $
Ext^{i}_{R}((R/\fa)^{s},X)\in \delta$ and by induction hypothesis
$Ext^{i-1}_{R}(K,X) \in \delta$ therefore $Ext^{i}_{R}(N,X) \in
\delta$. This terminates the induction argument. Now, an application
of Lemma 3.3 completes the proof.
\end{proof}
\end{proposition}


\begin{proposition}(compare [Mel; Corollary 3.14]) Suppose that $X$ is an arbitrary
$R$-module and that
 $cd(\fa,X)=1$. Then for all
$i\geq 0$, $Ext^{i}_{R}(R/\fa,H^{1}_{\fa}(X))\cong
Ext^{i+1}_{R}(R/\fa,X/\Gamma_{\fa}(X)).$ In particular, if $X$ is
finite, then $H^{i}_{\fa}(X)$ is $\fa$-cofinite for all $i.$

\begin{proof} Since $cd(\fa,X)=1$, $H^{i}_{\fa}(D_{\fa}(X))=0$
 for all $i\geq 0$. Hence $Ext^{i}_{R}(R/\fa,D_{\fa}(X))=0$  for all $i\geq
 0$  by Theorem 2.2. Now, by the ideal transform sequence
 $$0 \rightarrow X/\Gamma_{\fa}(X) \rightarrow  D_{\fa}(X) \rightarrow H^{1}_{\fa}(X) \rightarrow 0,$$
 we get $Ext^{i}_{R}(R/\fa,H^{1}_{\fa}(X))\cong
Ext^{i+1}_{R}(R/\fa,X/\Gamma_{\fa}(X))$ for all $i\geq 0$.
\end{proof}
\end{proposition}

The next proposition introduces some finite generalized local
cohomology modules.


\begin{proposition} Suppose that $M,N$ are finite $R$-modules
such that $M=\Gamma_{\fa}(M)$ and $cd(\fa,N)=1.$ Then
$H^{i}_{\fa}(M,N)$ is finite for all $i\geq 0$.

\begin{proof} Since $cd(\fa,N)=1$, $H^{i}_{\fa}(D_{\fa}(N))=0$
 for all $i\geq 0$. Hence $H^{i}_{\fa}(M,D_{\fa}(N))=0$
 whenever $i\geq 0$  by Proposition 2.3. Then the ideal transform sequence
 $$0 \rightarrow N/\Gamma_{\fa}(N) \rightarrow  D_{\fa}(N) \rightarrow H^{1}_{\fa}(N) \rightarrow
 0$$
 in conjunction with [YKS;
 Lemma 1.1] yield  the  isomorphisms below:\\
$H^{i+1}_{\fa}(M,N/\Gamma_{\fa}(N))\cong
H^{i}_{\fa}(M,H^{1}_{\fa}(N))\cong
 Ext^{i}_{R}(M,H^{1}_{\fa}(N))$ and $H^{i}_{\fa}(M,\Gamma_{\fa}(N))\cong
 Ext^{i}_{R}(M,\Gamma_{\fa}(N))$
for all $i\geq 0$. By Proposition 3.5 $\Gamma_{\fa}(N)$ and
$H^{1}_{\fa}(N)$ are
 $\fa$-cofinite and, thus, by Proposition 3.4 $Ext^{i}_{R}(M,H^{1}_{\fa}(N))$ and
 $Ext^{i}_{R}(M,\Gamma_{\fa}(N))$ are finite for all $i\geq 0$.
Thus
 $H^{i}_{\fa}(M,N/\Gamma_{\fa}(N))$ and $H^{i}_{\fa}(M,\Gamma_{\fa}(N))$
 are finite for all $i\geq 0$ which in turn shows that $H^{i}_{\fa}(M,N)$ is finite for all $i\geq
 0$.
\end{proof}
\end{proposition}

K.-I. Yoshida [Yo] and, D. Delfino and T. Marley [DM] proved that in
the case $R$ is local and $dimR/\fa=1$, $H^{i}_{\fa}(M)$ is
$\fa$-cofinite for all $i\geq 0$ and all finite $R$-modules $M$.
K.-I. Kawasaki [K] proved the same property of local cohomology
modules in the case where $\fa$ is principle.

As a natural generalization to the above facts, S. Yassemi [Y2;
Question 2.7] proposed a question whether the same results hold for
generalized local cohomology modules. This question was
affirmatively answered by K. Divaani-Aazar and R. Sazeedeh in [DS].
In the next theorem and its corollaries, we recover the main
theorems of [DS] in  Corollary 3.9 and Corollary 3.10; furthermore a
wider class of ideals and $R$-modules are provided for which the
corresponding generalized local cohomology modules are cofinite.


\begin{theorem}Suppose that $M,N$ are finite $R$-modules with
$pd(M)=t<\infty$ and $H^{i}_{\fa}(N)$ is $\fa$-cofinite for all
$i\geq 0$; suppose also that the category of $\fa$-cofinite
$R$-modules is Abelian. Then $H^{i}_{\fa}(M,N)$ is $\fa$-cofinite
for all $i\geq 0$.

\begin{proof} We prove by using induction on $t$. In the case where
$t=0$, $M$ is a projective $R$-module; so that there exist an
$R$-module $M'$ and an integer $s$ such that $M\oplus M'\cong
R^{s}$. Hence by Lemma 2.5, for each $i\geq 0$,
$$H^{i}_{\fa}(M,N)\oplus H^{i}_{\fa}(M',N)\cong Hom_{R}(M,H^{i}_{\fa}(N))\oplus Hom_{R}(M',H^{i}_{\fa}(N))$$
$$ \cong Hom_{R}(R^{s},H^{i}_{\fa}(N))\cong H^{i}_{\fa}(N)^{s}.$$
Which implies the assertion in this case.

Suppose that $t>0$ and that the assertion holds for all finite
$R$-modules with projective dimension less than $t$. Considering the
exact sequence
$$0 \rightarrow K \rightarrow  R^{n} \rightarrow M \rightarrow 0$$
where $K$ is an $R$-module and $n$ is an integer, and we obtain the
long exact sequence
$$\cdots \rightarrow H^{i-1}_{\fa}(R^{n},N) \rightarrow H^{i-1}_{\fa}(K,N) \rightarrow H^{i}_{\fa}(M,N) \rightarrow H^{i}_{\fa}(R^{n},N) \rightarrow H^{i}_{\fa}(K,N) \rightarrow \cdots.$$
By induction hypothesis
$H^{i-1}_{\fa}(R^{n},N),H^{i-1}_{\fa}(K,N),H^{i}_{\fa}(R^{n},N)$ and
$H^{i}_{\fa}(K,N)$ are $\fa$-cofinite while by our assumption, the
category of $\fa$-cofinite $R$-modules is Abelian. Therefore
$H^{i}_{\fa}(M,N)$ is $\fa$-cofinite as we desired.
\end{proof}
\end{theorem}

Applying [Mel; Theorems 7.4  and  7.10] and the above theorem it is
straightforward to deduce the following corollary.


\begin{corollary} Suppose that $dimR\leq 2$ and $M,N$ are finite
$R$-modules with $pd(M)<\infty$. Then $H^{i}_{\fa}(M,N)$ is
$\fa$-cofinite for all $i\geq 0$.
\end{corollary}

The next two corollaries have already been proved in [DS; Theorems
2.8 and 2.9]. Here we present new proofs based on Theorem 3.7.


\begin{corollary} Suppose that $\fa$ is a principle ideal. Then
for finite $R$-modules $M,N$ with $pd(M)<\infty$, $H^{i}_{\fa}(M,N)$
is $\fa$-cofinite for all $i\geq 0$.

\begin{proof} According to Proposition 3.5 and Theorem 3.7 it is enough to
prove that the category of $\fa$-cofinite $R$-modules is Abelian.

Let $f:L\rightarrow T$ be an $R$-homomorphism where $L,T$ are two
$\fa$-cofinite $R$-modules. To prove that the category of
$\fa$-cofinite $R$-modules is Abelian  we need only to show that
$I=imf$ is an $\fa$-cofinite $R$-module.

Let $K=kerf$ and $C=cokerf$. Considering the exact
sequences\\
\centerline{$0 \rightarrow I \rightarrow T \rightarrow C \rightarrow
0~~~$ and $~~~0 \rightarrow K \rightarrow L \rightarrow I
\rightarrow 0,$}
one can obtain  exact sequences\\
\centerline{$0 \rightarrow 0:_{I}\fa \rightarrow 0:_{T}\fa
\rightarrow 0:_{C}\fa~~~$ and $~~~K/\fa K \rightarrow L/\fa L
\rightarrow I/\fa I \rightarrow 0.$} The latter exact sequences
imply  that $0:_{I}\fa$ and $I/\fa I$ are finite.
 Therefore, by [Mel; Corollary 3.4], $I$
is $\fa$-cofinite as we desired.
\end{proof}
\end{corollary}


\begin{corollary} Suppose that $\fp$ is a prime ideal of a
complete local ring $(R,\fm)$ with $dimR/\fp=1$. Then for finite
$R$-modules $M,N$ with $pd(M)<\infty$, $H^{i}_{\fp}(M,N)$ is
$\fp$-cofinite for all $i\geq 0$.

\begin{proof} Our assumption on $\fp$ in conjunction with
[DM; Theorem 2] imply that the category of $\fp$-cofinite
$R$-modules is Abelian; so that this corollary is an consequence of
Theorem 3.7 and the mentioned theorem of Yoshida  and, Delfino and
Marley.
\end{proof}
\end{corollary}


\begin{corollary} Suppose that $(R,\fm)$ is a local ring such
that $R/\fa$ is analytically irreducible (for example regular) of
dimension one. Then for finite $R$-modules $M,N$ with
$pd(M)<\infty$,  $H^{i}_{\fa}(M,N)$ is $\fa$-cofinite for all $i\geq
0$.

\begin{proof}
As $\widehat{R}/\widehat{\fa}$ is a complete local domain, by
Corollary 3.10 $H^{i}_{\widehat{\fa}}(M\otimes_{R}
\widehat{R},N\otimes_{R} \widehat{R})$ and thus, by [DST; Lemma 2.1
(ii)], $H^{i}_{\fa}(M,N)\otimes_{R} \widehat{R}$ is
$\widehat{\fa}$-cofinite for all $i\geq 0$. Now, [Mel2; Theorem
11.5(ii)] implies that $H^{i}_{\fa}(M,N)$ is $\fa$-cofinite for all
$i\geq 0$.
\end{proof}
\end{corollary}


\begin{corollary} Let $(R,\fm)$ be a complete local ring and
$\fa$ be an ideal of $R$ with $dimR/\fa=1$. Suppose also that for
each $\fa$-cofinite $R$-module $L$ and each minimal member of
$V(\fa)$, say $\fp$, $\Gamma_{\fp}(L)$ and $H^{1}_{\fp}(L)$ are
$\fp$-cofinite. Then $H^{i}_{\fa}(M,N)$ is $\fa$-cofinite for all
$i\geq 0$ whenever $M,N$ are finite $R$-modules with $pd(M)<\infty$.

\begin{proof} By [Mel; Proposition 7.12], our assumption on $\fa$ implies
that the category of $\fa$-cofinite $R$-modules is Abelian. Now, the
assertion follows from Theorem 3.7 and the mentioned theorem of
Yoshida and, Delfino and Marley.
\end{proof}
\end{corollary}


{\bf Acknowledgments.}  The authors wish to thank the referee for
his/her suggestions which brought improvement in the presentation of
this paper.


\providecommand{\bysame}{\leavevmode\hbox
to3em{\hrulefill}\thinspace}

\end{document}